\documentclass[12pt]{amsart}
\usepackage{colortbl}
\usepackage{longtable}
\usepackage{hyperref}
\usepackage{amsmath}
\usepackage{pslatex}
\usepackage{amsthm}
\usepackage{amssymb}
\usepackage{latexsym}
\usepackage{graphicx}
\usepackage{multicol}
\usepackage{tikz}
\usepackage{pst-plot}
\usepackage{pstricks}

\hypersetup{hypertex, colorlinks=true, linkcolor=blue, filecolor=blue, pagecolor=blue, urlcolor=blue}

\newtheorem{theorem}{Theorem}
\newtheorem{lemma}[theorem]{Lemma}
\newtheorem{corollary}[theorem]{Corollary}

\begin{document}

\author{Skyler Simmons}
\address{Department of Mathematics, Utah Valley University, Orem, UT 84058}
\email{skyler.simmons@uvu.edu}

\title{The Octahedral Collision-Based Periodic Orbit in the Three-Dimensional Six-Body Problem}

\begin{abstract}
We construct a highly-symmetric periodic orbit of six bodies in three dimensions.  In this orbit, binary collisions occur at the origin in a regular periodic fashion, rotating between pairs of bodies located on the coordinate axes.  Regularization of the collisions in the orbit is achieved by an extension of the Levi-Civita method.  Initial conditions for the orbit are found numerically.  In contrast to an earlier periodic collision-based orbit in three dimensions, this orbit is shown to be unstable.
\end{abstract}


\maketitle

\section{Introduction}
In the \textit{Principia Mathematica} (see \cite{bibNewton}), Newton gives mathematical equations governing the motion of point masses within their mutual gravitational field.  Specifically, for $n$ point masses in $\mathbb{R}^d$ located at $\mathbf{x}_i$ with mass $m_i$ for $i = 1, 2, ..., n$, we have that
\begin{equation}
\label{NewtonEquation}
m_i\ddot{\mathbf{x}}_i = \sum_{i \neq j} \frac{Gm_im_j}{|\mathbf{x}_i - \mathbf{x}_j|^2} \left(\frac{\mathbf{x}_i - \mathbf{x}_j}{|\mathbf{x}_i - \mathbf{x}_j|}\right).
\end{equation}
Here, the dot represents the derivative with respect to time, and $G$ is a constant.  A suitable choice of units gives $G = 1$, which is often assumed for mathematical simplicity.  (In SI units, the US National Institute of Standards and Technology\footnote{https://physics.nist.gov/cuu/Constants/index.html} gives the value $G = 6.67430 \times 10^{-11} \text{m}^3 / \text{kg} \cdot \text{s}^2$, with a standard uncertainty of $0.00015 \times 10^{-11} \text{m}^3 / \text{kg} \cdot \text{s}^2$.)\\

\textit{Collision singularities} of the $n$-body problem occur when $\mathbf{x}_i = \mathbf{x}_j$ for some $i \neq j$.  Under suitable conditions, collisions of two bodies can be regularized.  \textit{Regularization} involves a change of temporal and spatial variables so that the collision point becomes a regular point for the resulting ODEs.  Collision singularities have received a great deal of study.  Of particular note is a result by McGehee \cite{bibMcGehee1}, which shows that in general, a collision of three or more bodies cannot be regularized. \\

Many periodic orbits featuring collisions have been produced.  Existence, stability, and other properties of periodic orbits with three bodies in one spatial dimension are studied in both analytical and numerical contexts as early as 1956 in \cite{bibJS} and as recently as 2019 in \cite{bibKTXY}.  Works between these years include \cite{bibHE}, \cite{bibHM}, \cite{bibST1}, \cite{bibMoeckel}, \cite{bibVE}, \cite{bibST2}, \cite{bibST3}, \cite{bibShib}, \cite{bibYan3}, and \cite{bibOY}.  Orbits with four bodies in one spatial dimension are featured in \cite{bibShib}, \cite{bibMartinez}, \cite{bibHuang} and \cite{bibYan1}.  Orbits in two spatial dimensions featuring collisions were studied as early as 1979 in \cite{bibBroucke} and as recently as 2021 in \cite{bibSim}, with other notable works including \cite{bibRoySteves}, \cite{bibBOYSR}, \cite{bibSSS}, \cite{bibBOYS}, \cite{bibWaldvogel}, \cite{bibBMS}, \cite{bibOY3}, \cite{bibYan2}, and \cite{bibBS1}.  Additionally, in \cite{bibShib} and \cite{bibMartinez}, large families of highly-symmetric orbits are given in one, two and three dimensions, all of which can be expressed in two degrees of freedom. Three-dimensional restricted collision-based orbits are studied in \cite{bibMISC}, \cite{bibBV}, \cite{bibBDV}, and \cite{bibGPSV} as a case of the $e = 1$ Sitnikov problem, which can be reduced to a time-dependent two-degree-of-freedom problem.\\

This paper is similar in content to \cite{bibSimCubic}.  Indeed, many of the arguments presented herein are identical, or nearly identical, to those in \cite{bibSimCubic}.  We study a three-degree of freedom, highly-symmetric, periodic orbit of six bodies featuring collisions.  Two bodies lie on each of the three coordinate axes in $\mathbb{R}^3$, symetrically located relative to the origin.  Overall, the six bodies form the vertices of an octahedron with eight congruent (often scalene) triangular faces at all points in time.  Each body collides in turn with its ``companion'' body lying on the same coordinate axis.  \\

The remainder of the paper is as follows: In Section \ref{secSetting}, we set up and regularize the Hamiltonian that corresponds to the configuration being considered.  Section \ref{secPeriodicOrbit} details the construction of the periodic orbit.  We first describe the orbit in the regularized setting.  Then, we analytically establish sufficient conditions for the orbit to exist.  Finally, we complete the existence proof with some numerical calculations. \\

Section \ref{secStabSim} discusses the linear stability of the orbit.  We first review some preliminary details of stability.  Next, we establish notation for the symmetries of the orbit.  We next detail some results by Roberts in \cite{bibRoberts8} that allow us to determine the linear stability of the orbit in a rigorous numerical fashion in terms of these symmetries.  Applications to the orbit under consideration are detailed after each result.  Finally, in Section \ref{secStabRes}, we give results of the numerical stability calculation established in the previous section and show that this orbit is linearly unstable. \\

\section{The Hamiltonian Setting and Regularization}
\label{secSetting}

\subsection{Configuration}
\label{subConfig}
We consider the Newtonian 6-body problem with point unit masses located at $(\pm q_1, 0, 0)$, $(0, \pm q_2, 0)$, and $(0,0,\pm q_3)$ (see Figure \ref{OctahedralConfig}).  At all times where $q_1q_2q_3 \neq 0$, the bodies lie at the vertices of an octahedron with 8 congruent (and usually scalene) triangular faces.  We will accordingly refer to this as the \textit{octahedral configuration}. \\

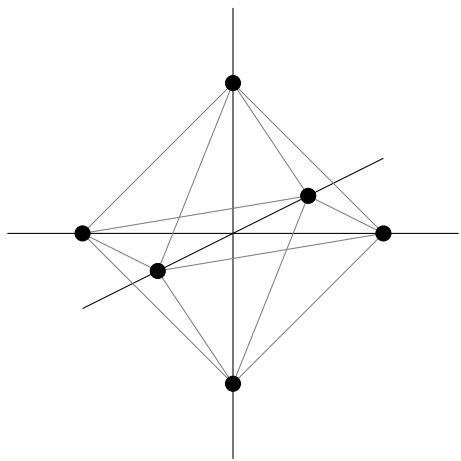
\begin{figure}[h]

\begin{tikzpicture}[scale = 1]
\draw[-] (-3,0) -- (3,0);
\draw[-] (0,-3) -- (0,3);
\draw[-] (-2,-1)--(2,1);
\draw[-, gray] (-2,0)--(0,-2);
\draw[-, gray] (2,0)--(0,2);
\draw[-, gray] (2,0)--(0,-2);
\draw[-, gray] (-2,0)--(0,2);
\draw[-, gray] (-2,0)--(-1,-.5);
\draw[-, gray] (-2,0)--(1,.5);
\draw[-, gray] (2,0)--(-1,-.5);
\draw[-, gray] (2,0)--(1,.5);
\draw[-, gray] (0,-2)--(-1,-.5);
\draw[-, gray] (0,-2)--(1,.5);
\draw[-, gray] (0,2)--(-1,-.5);
\draw[-, gray] (0,2)--(1,.5);

\draw[fill] (-2,0) circle [radius = 0.1];
\draw[fill] (2,0) circle [radius = 0.1];
\draw[fill] (0,-2) circle [radius = 0.1];
\draw[fill] (0,2) circle [radius = 0.1];
\draw[fill] (-1,-.5) circle [radius = 0.1];
\draw[fill] (1,.5) circle [radius = 0.1];
\end{tikzpicture}

\caption{The Octahedral Configuration}
\label{OctahedralConfig}
\end{figure} 

Note that for distinct subscripts $i, j, k \in \{1,2,3\}$, when $q_i = 0$, if $q_jq_k \neq 0$, a binary collision occurs at the origin.  We seek an orbit in which binary collisions rotate through the pairs of bodies on the $x$, $y$, and $z$ axes in a periodic fashion as pictured in Figure \ref{PeriodicOrbitFigure}.  Moreover, the amount of time between successive collisions is the same throughout. \\

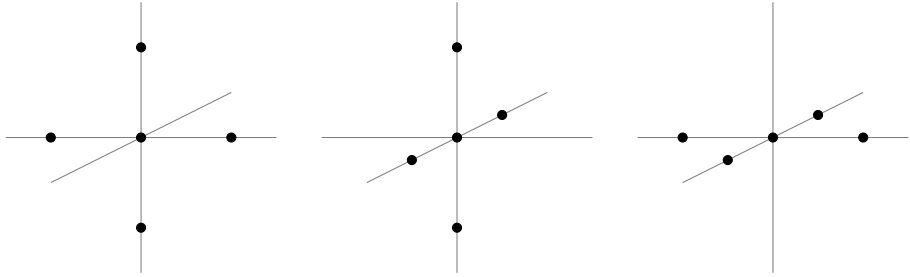
\begin{figure}[h]
\begin{tikzpicture}[scale = .6]

\draw[-, gray] (-10,0) -- (-4,0);
\draw[-, gray] (-7,-3) -- (-7,3);
\draw[-, gray] (-9,-1)--(-5,1);
\draw[fill] (-9,0) circle [radius = 0.1];
\draw[fill] (-5,0) circle [radius = 0.1];
\draw[fill] (-7,-2) circle [radius = 0.1];
\draw[fill] (-7,2) circle [radius = 0.1];
\draw[fill] (-7,0) circle [radius = 0.1];

\draw[-, gray] (-3,0) -- (3,0);
\draw[-, gray] (0,-3) -- (0,3);
\draw[-, gray] (-2,-1)--(2,1);
\draw[fill] (0,-2) circle [radius = 0.1];
\draw[fill] (0,2) circle [radius = 0.1];
\draw[fill] (-1,-.5) circle [radius = 0.1];
\draw[fill] (1,.5) circle [radius = 0.1];
\draw[fill] (0,0) circle [radius = 0.1];

\draw[-, gray] (4,0) -- (10,0);
\draw[-, gray] (7,-3) -- (7,3);
\draw[-, gray] (5,-1)--(9,1);
\draw[fill] (9,0) circle [radius = 0.1];
\draw[fill] (5,0) circle [radius = 0.1];
\draw[fill] (6,-.5) circle [radius = 0.1];
\draw[fill] (8,.5) circle [radius = 0.1];
\draw[fill] (7,0) circle [radius = 0.1];

\end{tikzpicture}

\caption{The periodic orbit.  Binary collisions occur at regular time intervals between the pairs of bodies on the $x$, $y$, and $z$ axes.}
\label{PeriodicOrbitFigure}
\end{figure}

An analogous two-degree-of-freedom orbit of four bodies in the $xy$ plane, with bodies located at $(\pm x,0)$ and $\pm (0,\pm y)$ featuring alternating collisions at the origin was shown to exist in \cite{bibYan1}, in which linear stability of the orbit in the equal-mass case was also shown.  Linear stability of the orbit in the pairwise-equal mass case was shown in \cite{bibBS1}, in which the stability is shown to depend on the ratio of the two mass pairs.  Additionally, in \cite{bibBS1}, the orbit is shown to be unstable if the bodies are allowed to drift off of the coordinate axes. \\

\subsection{The Hamiltonian Setting}
\label{subHamiltonian}
The potential energy is the sum of 15 terms.  Twelve of these terms arise correspond to the edges of the octohedron, and have the form
\begin{equation}
\frac{1}{\sqrt{q_i^2 + q_j^2}}.
\end{equation}
For each of the three possible pairs of subscripts, there are four such terms.  The interaction of each body with its ``companion'' on the same axis gives the final three:
\begin{equation}
\frac{1}{\sqrt{(2q_i)^2}} = \frac{1}{2q_i}.
\end{equation}
So the potential energy of the system is
\begin{equation}
U = \frac{4}{\sqrt{q_1^2 + q_2^2}} + \frac{4}{\sqrt{q_1^2 + q_3^2}} + \frac{4}{\sqrt{q_2^2 + q_3^2}} + \frac{1}{2q_1} + \frac{1}{2q_2} + \frac{1}{2q_3}.
\end{equation}

Let $p_i = \dot{q}_i$ denote the components of the momentum of the bodies.  The kinetic energy for the system is 
\begin{equation}
K = 2\left(\frac{p_1^2}{2}\right) + 2\left(\frac{p_2^2}{2}\right) + 2\left(\frac{p_3^2}{2}\right) = p_1^2 + p_2^2 + p_3^2.
\end{equation}
The Hamiltonian for the system is then given by $H = K - U$. \\


\subsection{Regularization}
\label{subReg}
We regularize the collisions that occur at $q_i = 0$ using an extension of the Levi-Civita method (see \cite{bibLeviCivita}).  Specifically, let
\begin{equation}
F = \sum_{i = 1}^3 \sqrt{q_i}P_i.
\end{equation}
This generates a coordinate transformation given by
\begin{equation}
Q_i = \frac{\partial F}{\partial P_i} = \sqrt{q_i}, \quad p_i = \frac{\partial F}{\partial q_i} = \frac{P_i}{2\sqrt{q_i}},
\end{equation}
or
\begin{equation}
q_i = Q_i^2, \quad p_i = \frac{P_i}{2Q_i}.
\end{equation}

In these coordinates, the potential energy for the system is given by
\begin{equation}
\tilde{U} = \frac{4}{\sqrt{Q_1^4 + Q_2^4}} + \frac{4}{\sqrt{Q_1^4 + Q_3^4}} + \frac{4}{\sqrt{Q_2^4 + Q_3^4}} + \frac{1}{Q_1^2} + \frac{1}{Q_2^2} + \frac{1}{Q_3^2}
\end{equation}

The new kinetic energy is given by
\begin{equation}
\tilde{K} = \frac{P_1^2}{4Q_1^2} + \frac{P_2^2}{4Q_2^2} + \frac{P_3^2}{4Q_3^2}.
\end{equation}
The new Hamiltonian is given by $\tilde{H} = \tilde{K} - \tilde{U}.$ \\

Lastly, to regularize the collisions at $Q_i = 0$, we apply a change of time satisfying
\begin{equation}
\label{tsrelation}
\frac{dt}{ds} = Q_1^2Q_2^2Q_3^2.
\end{equation}
This gives the regularized Hamiltonian $\Gamma = \frac{dt}{ds}(\tilde{H} - E),$ or
\begin{align}
\label{regHam}
\Gamma &= \frac{P_1^2Q_2^2Q_3^2}{4} + \frac{Q_1^2P_2^2Q_3^2}{4} + \frac{Q_1^2Q_2^2P_3^2}{4} \nonumber \\ 
&- \frac{4Q_1^2Q_2^2Q_3^2}{\sqrt{Q_1^4 + Q_2^4}} - \frac{4Q_1^2Q_2^2Q_3^2}{\sqrt{Q_1^4 + Q_3^4}} - \frac{4Q_1^2Q_2^2Q_3^2}{\sqrt{Q_2^4 + Q_3^4}} \\
&- \frac{Q_2^2Q_3^2}{2} - \frac{Q_1^2Q_3^2}{2} - \frac{Q_1^2Q_2^2}{2} - Q_1^2Q_2^2Q_3^2E, \nonumber
\end{align}
where $E$ is the fixed energy of the system. \\

We now show that the system has been regularized as claimed.  Let $i, j, k$ be distinct.  Then, at the collision where $Q_i = 0$, $Q_j \neq 0$, and $Q_k \neq 0$, the condition $\Gamma = 0$ forces
\begin{equation}
\label{regularized}
\frac{P_i^2Q_j^2Q_k^2}{4} - \frac{Q_j^2Q_k^2}{2} = 0.
\end{equation}
Then $P_i = \pm \sqrt{2}$.  Moreover, since 
\begin{equation}
\label{collisionCondition}
\dot{Q}_i = \frac{d\Gamma}{dP_i} = \frac{P_iQ_j^2Q_k^2}{2}
\end{equation}
then $\dot{Q}_i \neq 0$ when $Q_i = 0$.  Hence the orbit can be continued past the collision.  (In the regularized setting, we will use the dot notation to represent the derivative with respect to the new time variable $s$.) \\

An important feature of the regularization that can be determined from Equation \ref{collisionCondition} is that both $\dot{Q}_i$ and $P_i$ have the same sign at the collision time.  Since $P_i$ is continuous and non-zero at collision time, the sign of $P_i$ is the same before and after the collision.  Hence, the sign of $\dot{Q}_i$ also does not change, so $Q_i$ must either pass from a negative to a positive value at collision, or from a positive to a negative one.  This is a characteristic feature of the Levi-Civita regularization, where the original spatial variables are double-covered after the regularization.\\

It is worth noting that the planar orbit discussed in \cite{bibYan1} and \cite{bibBS1} does not exist as an invariant subspace or the octahedral orbit under consideration. \\

\section{The Periodic Orbit}
\label{secPeriodicOrbit}
\subsection{Description}
\label{subDescription}
The desired orbit passes through binary collisions at the origin caused by $Q_1 = 0$, $Q_2 = 0$, and then $Q_3 = 0$ in a cyclic fashion, as pictured in Figure \ref{PeriodicOrbitFigure}.  In a physical sense, we start with the bodies with (non-regularized) positions given by
\begin{align}
\pm q_1 &= 0, \nonumber \\
\pm q_2 &= \omega, \\
\pm q_3 &= \omega, \nonumber
\end{align}
and ending at
\begin{align}
\pm q_1 &= \omega, \nonumber \\
\pm q_2 &= 0, \\ 
\pm q_3 &= \omega, \nonumber
\end{align}
for some positive number $\omega$.  In other words, the bodies start at binary collision of the $x$-axis bodies with the remaining four bodies forming a square in the $yz$ plane.  After some amount of time, the two $y$-axis bodies will collide at the origin with the other four bodies forming a square in the $xz$ plane.  The proposed orbit will then be extended by a symmetry coinciding with a rotation of $120^\circ$ about the line $x = y = z$ in $\mathbb{R}^3$.  That is to say, the orbit continues through a sequence of collisions 
\begin{equation}
(\pm q_1, \pm q_2, \pm q_3):(0, \omega, \omega) \to (\omega, 0, \omega) \to (\omega, \omega, 0) \to (0, \omega, \omega) \to \ldots,
\end{equation}
with the collisions being equally-spaced in time. \\

In the regularized coordinates, the velocity components can also be defined.  Let $\gamma(s) = (Q_1(s), Q_2(s), Q_3(s), P_1(s), P_2(s), P_3(s))^T$.  At each collision time with $Q_i = 0$, the sign of $Q_i$ changes as noted at the end of Section \ref{subReg}.  Additionally, both $\gamma(s)$ and $-\gamma(s)$ correspond to the same setting in the original coordinates.  Hence, in the regularized setting, one period of the orbit passes through six collisions rather than three.

\subsection{Extension by Symmetry}
\label{subExtension}

\begin{lemma}
\label{symmetryconstruct}
Suppose $\gamma(s)$ is a solution to the regularized Hamiltonian system $\Gamma$ that satisfies
\begin{equation}
\label{initialGamma}
\gamma(0) = (0, \alpha, \alpha, \sqrt{2}, -\beta, \beta)^T \\
\end{equation}
and
\begin{equation}
\label{finalGamma}
\gamma(2\tau) = (\alpha, 0, \alpha, \beta, -\sqrt{2}, -\beta)^T.
\end{equation}
for some $\tau > 0$ and $E < 0$.  Then $\gamma(s)$ extends to a $12\tau$ periodic orbit for the system $\Gamma$.
\end{lemma}

\begin{proof}
We first establish the symmetries that will allow us to extend the orbit as claimed.  By direct calculation, we find that the equation for $\dot{Q}_i$ is negated under the transformation $P_i \mapsto -P_i$ and remains fixed under any sign change of the remaining variables.  We also find that $\dot{P}_i$ is negated under $Q_i \mapsto -Q_i$ and remains fixed under any other sign change of the remaining variables.  Furthermore, for any permutation $\sigma \in S_3$, since $\Gamma$ is fixed under permutation of the subscripts by $\sigma$, then the equation of motion for $\dot{Q}_i$ in terms of $Q_1, Q_2, Q_3, P_1, P_2, P_3$ is the same as that of $\dot{Q}_{\sigma(i)}$ in terms of $Q_{\sigma(1)}, Q_{\sigma(2)}, Q_{\sigma(3)}, P_{\sigma(1)}, P_{\sigma(2)}, P_{\sigma(3)}$.  Similar permutation results hold for $P_i$. \\

Consider the orbit with initial conditions $\gamma(2\tau)$.  By the symmetries just discussed, we have that
\begin{align}
\dot{Q}_1(2\tau) &= \dot{Q}_3(0)\textcolor{red}{,} \nonumber \\
\dot{Q}_2(2\tau) &= -\dot{Q}_1(0), \nonumber \\
\dot{Q}_3(2\tau) &= \dot{Q}_2(0), \\
\dot{P}_1(2\tau) &= \dot{P}_3(0), \nonumber \\
\dot{P}_2(2\tau) &= -\dot{P}_1(0), \nonumber \\
\dot{P}_3(2\tau) &= \dot{P}_2(0)\textcolor{red}{.} \nonumber
\end{align}
Moreover, the equations of motion are the same as those on the interval $s \in [0, 2\tau]$ under permutations and sign changes as discussed above.  Existence and uniqueness of solutions to differential equations gives
\begin{align}
Q_1(s + 2\tau) &= Q_3(s), \nonumber \\
Q_2(s + 2\tau) &= -Q_1(s), \nonumber \\
Q_3(s + 2\tau) &= Q_2(s), \\
P_1(s + 2\tau) &= P_3(s), \nonumber \\
P_2(s + 2\tau) &= -P_1(s), \nonumber \\
P_3(s + 2\tau) &= P_2(s), \nonumber
\end{align}
is a solution to the Hamiltonian system given by $\Gamma$.  Setting $s = 2\tau$, we have that
\begin{equation}
\gamma(4\tau) = (\alpha, -\alpha, 0, -\beta, -\beta, -\sqrt{2})^T.
\end{equation}
Repeating the argument with initial conditions given by $\gamma(4\tau)$ gives
\begin{equation}
\gamma(6\tau) = (0, -\alpha, -\alpha, -\sqrt{2}, \beta, -\beta)^T.
\end{equation}
Continuing in turn, we have that
\begin{align}
\gamma(8\tau) &= (-\alpha, 0, -\alpha, -\beta, \sqrt{2}, \beta)^T, \\
\gamma(10\tau) &=(-\alpha, \alpha, 0, \beta, \beta, \sqrt{2})^T, \\
\gamma(12\tau) &=(0, \alpha, \alpha, \sqrt{2}, -\beta, \beta)^T. 
\end{align}
Since $\gamma(0) = \gamma(12\tau)$, the periodic orbit has been constructed as claimed.
\end{proof}

Physically, the portion of the orbit assumed to exist in Lemma \ref{symmetryconstruct} corresponds to an orbit in which all bodies start in the $x = 0$ plane with a binary collision at the origin and the remaining bodies forming a square along the $y$ and $z$ coordinate axes.  At that time, the bodies along the $y$ axis move towards the origin, while the bodies on the $x$ and $z$ axes move away from it.  As the $y$-axis bodies move towards collision, the $z$-axis bodies move outward, eventually slowing down, stopping, and turning around.  At the time of the collision of the $y$-axis bodies, the bodies on the $x$ and $z$ axes again form a square, with the $z$ axis bodies lying in their same initial position with their velocity now reversed.  \\

\textbf{Note:} We do not rule out the possibility of the existence of a ``less symmetric'' orbit.  Indeed, the arguments in Lemma \ref{symmetryconstruct} give the same conclusion if we assume that
\begin{align}
\gamma(0) &= (0, a, b, \sqrt{2}, -c, d)^T,\\
\gamma(2\tau) &= (b, 0, a, d, -\sqrt{2}, -c)^T,
\end{align}
without the requirement that $a = b$ and $c = d$.  However, for simplicity we restrict ourselves to the ``reduced'' case at the present time.\\

\begin{lemma}
\label{symmetryconstruct2}
Suppose $\gamma(s)$ is a solution to the regularized Hamiltonian system $\Gamma$ with $\gamma(0)$ defined as in Equation \ref{initialGamma} and that satisfies
\begin{equation}
\label{initialGamma2}
\gamma(\tau) = (a, a, b, c, -c, 0)^T.
\end{equation}
for some $\tau > 0$ and $E < 0$.  Then $\gamma(s)$ extends to a $12\tau$ periodic orbit for the system $\Gamma$.
\end{lemma}
\begin{proof}
Suppose $\gamma(s)$ exists as claimed.  Consider an orbit with initial conditions
\begin{equation}
\label{reversedIconds}
\gamma(\tau) = (a, a, b, -c, c, 0)^T.
\end{equation}
Note that the values of all $P_i(\tau)$ have been negated.  By the chain rule, negating the momentum terms is equivalent to reversing time.  Then for the initial conditions in Equation \ref{reversedIconds}, existence and uniqueness must give us that
\begin{equation}
\gamma(2\tau) = (0, \alpha, \alpha, -\sqrt{2}, \beta, -\beta)^T. \\
\end{equation}
Then using a similar symmetry argument as in Lemma \ref{symmetryconstruct}, it must be that if
\begin{equation}
\gamma(\tau) = (a, a, b, c, -c, 0)^T,
\end{equation}
then
\begin{equation}
\gamma(2\tau) = (\alpha, 0, \alpha, \beta, -\sqrt{2}, -\beta)^T.
\end{equation}
Applying Lemma \ref{symmetryconstruct} then gives the final result.
\end{proof}

Hence, in the regularized setting, the orbit can be constructed if we can find initial conditions that correspond to just the first twelfth of the orbit.  This 12-fold symmetry is similar to that of the figure-eight orbit of Moore, Chenciner, and Montgomery (see \cite{bibMoore} and \cite{bibChencinerMontgomery}).  \\

Note that at the time $\gamma(\tau)$, the vectors $\langle a, a, b \rangle$ and $\langle c, -c, 0 \rangle$ (corresponding to position and velocity, respectively) are orthogonal.  This type of extension by time-reversing symmetry and orthogonality has been studied in other works as well (examples include \cite{bibMartinez} and \cite{bibShib}). \\

\subsection{Existence of the Orbit}
\label{subExistence}

The results in this section serve to justify the numerical calculations that will ultimately be used to find the periodic orbit.  For some of the proofs in this section, it will be more convenient to use the original $q_i$ and $p_i$ coordinates. \\

\begin{lemma}
\label{CollisionGuaranteed}
Suppose at some time $s_0$ we have that $Q_iP_i < 0$.  Then if the orbit does not pass through an unregularized collision, there is some future time $s^* > s_0$ where $Q_i(s^*) = 0$.
\end{lemma}
\begin{proof}
This is most easily proven in the original $q$ and $p$ coordinates.  The initial conditions correspond to some $q_i > 0$ and $p_i < 0$.  Since $\dot{p}_i = \ddot{q}_i < 0$, then $p_i < 0$ as long as $q_i > 0$.  Let $t_0$ be the time corresponding to $s_0$.  For an interval of time containing $t_0$, we then have that
\begin{equation}
q_i(t) \leq p_i(t_0)(t - t_0) + q_i(t_0).
\end{equation}
Then the collision must occur before the time $t^*$ defined by
\begin{equation}
t^* = t_0 - \frac{q_i(t_0)}{p_i(t_0)}.
\end{equation}
By assumption, since no unregularized collision occurs along the orbit, there must be some value of $s^*$ corresponding to $t^*$, where the relation between $t^*$ and $s^*$ is given by Equation \ref{tsrelation}. \\
\end{proof}

To help facilitate future discussion, define $\Sigma$ to be the set of ordered pairs $(\alpha, \beta)$ satisfying:
\begin{itemize}
\item $\alpha \geq 0$
\item $\beta \geq 0$
\item If $\gamma(0)$ is defined as in Equation \ref{initialGamma} for $(\alpha, \beta) \in \Sigma$, then the minimum $s^* > 0$ for which $Q_1Q_2Q_3(s^*) = 0$ satisfies $Q_2(s^*) = 0$.
\end{itemize}

In Section \ref{subNumerICond}, our results will show some of the properties of the set $\Sigma$.  Non-numerically, we can say that by Lemma \ref{CollisionGuaranteed}, using the initial conditions as given in Equation \ref{initialGamma} with $(\alpha, \beta) \in \Sigma$, there must be \textit{some} future time where the $y$-axis bodies collide.  We wish to focus on the set of initial conditions where this collision is ``second'' (in the sense that at $s = 0$, the $x$-axis bodies are at the ``first'' collision). \\

It is natural to believe that for initial conditions given from Equation \ref{initialGamma}, if $\alpha \geq 0$ and $\beta \geq 0$, then the $z$-axis bodies cannot be the next collision.  Numerical integration of the equations of motion shows that this is indeed the case.  Some analytical results towards this end will be shown in Section \ref{subNumerICond}.  \\

We recall again that at collision time, Equation \ref{regularized} holds.  Hence in the equations of motion, when $\alpha = 0$, there is no restriction on $\beta$.  A higher value of $\beta$ corresponds to a greater velocity of the non-colliding masses.  Therefore, given fixed positions, in order to preserve the constant energy of the system, the ``ejection velocity'' of the colliding bodies must decrease as $\beta$ increases.  (Although the value of $P_1(0) = \pm \sqrt{2}$ regardless, the value of $|P_1(s)|$ must be near zero shortly after the collision.)  Hence for a fixed value of $\alpha$ and a sufficiently high value of $\beta$, the first collision after $s = 0$ must again be a collision of the $x$-axis bodies.  We can therefore conclude that for any value of $\alpha$, the values of $\beta$ for which $(\alpha, \beta) \in \Sigma$ must be bounded.  Furthermore, at this bounding value of $\beta$, the orbit ends in a quadruple collision $Q_1 = Q_2 = 0$ after infinite regularized time.\\

\begin{lemma}
\label{alphastarlemma}
For a fixed $E < 0$, there exists some $\alpha^* > 0$ so that if $\alpha > \alpha^*$, then $(\alpha, 0) \not \in \Sigma$.
\end{lemma}
\begin{proof}
As $\alpha \to \infty$, with $\gamma(0)$ as defined in Equation \ref{initialGamma}, the configuration approaches a single binary pair of bodies colliding in periodic fashion on the $x$ axis and four decoupled bodies on the $y$ and $z$ axes.  Certainly for $\alpha$ sufficiently large it must be the case that the first collision in positive time occurs at $Q_1 = 0$, as the gravitational force of the remaining four bodies becomes negligible.  For such $\alpha$, it is thus the case that $(\alpha, 0) \not \in \Sigma$.  Hence the set of all $\alpha$ for which $(\alpha, 0) \in \Sigma$ is bounded above.  We can then define $\alpha^*$ to be the supremum of this set.
\end{proof}

\begin{lemma}
\label{Existence1}
Let $\gamma(0)$ be defined as in Equation \ref{initialGamma} with $(\alpha, \beta) \in \Sigma$ and $E < 0$.  Then there is some $\tau > 0$ with $\tau = \tau(\alpha, \beta)$ so that $Q_1(\tau) = Q_2(\tau)$.
\end{lemma}

\begin{proof}
For some $s>0$ sufficiently small, we have that the signs of the components of $\gamma(s)$ are given by
\begin{equation}
\gamma(s) = (+, +, +, +, -, +).
\end{equation}
Then by Lemma \ref{CollisionGuaranteed}, we have that at some future time, $Q_2(s^*) = 0$.  Since $(\alpha, \beta) \in \Sigma$, it must be that $Q_1(s^*) \geq 0$.  Since $Q_2(0) - Q_1(0) = \alpha > 0$ and $Q_2(s^*) - Q_1(s^*) \leq 0$ the result follows by the Intermediate Value Theorem.
\end{proof}

Physically, the conditions at time $\tau$ correspond to the bodies on the $x$- and $y$-axes lying in a square in the $xy$ plane with the third pair of bodies symmetrically placed on the $z$ axis about the origin.  We need next to simultaneously control the velocity of the third pair of bodies on the $z$ axis and the velocities of the bodies on the $x$ and $y$ axes to establish the conditions set forth in Lemma \ref{symmetryconstruct2}.   \\

\begin{lemma}
\label{Existence2}
Let $(\alpha, \beta_0) \in \Sigma$ with $\alpha < \alpha^*$ and $P_3(\tau) > 0$, with $\tau$ as defined in Lemma \ref{Existence1}.  Then there is some $\beta > 0$ so that $(\alpha, \beta) \in \Sigma$ and
\begin{equation}
\label{existence2b}
\gamma(\tau) = (a, a, b, c_1, c_2, 0)^T.
\end{equation}
\end{lemma}
\begin{proof}
Consider $\gamma(0) = (0,\alpha, \alpha, \sqrt{2}, 0, 0)$.  Note that this is Equation \ref{initialGamma} with $\beta = 0$.  Then by the symmetry of $\Gamma$, $Q_2(s)=Q_3(s)$ and $P_2(s) = P_3(s)$.  For $s > 0$, $P_2(s)$ and $P_3(s)$ are both negative, so by Lemma \ref{CollisionGuaranteed} some future time $s^*$ it must be that $Q_2(s^*) = Q_3(s^*) < \epsilon$ for any $\epsilon > 0$.  (Note that this tends toward an unregularized quadruple collision $Q_2 = Q_3 = 0$ after infinite regularized time.)  Since $(\alpha, 0) \in \Sigma$ by supposition, we must have that $Q_1 > 0$ for all $0 < s < s^*$.  So at the time $Q_1(\tau) = Q_2(\tau)$, it must be that $P_3(\tau) < 0$.  On the other hand, by supposition, setting initial conditions corresponding to $(\alpha, \beta_0)$, then $Q_3(\tau)$ must also be positive.  The result then follows from the Intermediate Value Theorem. \\
\end{proof}

It may be the case that for some values of $\alpha$, the orbit corresponding to the value of $\beta$ on the boundary of $\Sigma$ has $P_3(\tau) < 0$.  These values of $\alpha$ cannot correspond to the periodic orbit.  We will see that this is sometimes the case in Section \ref{subNumerICond}. \\

We expect that away from the potentially chaotic region near $||\gamma(0)|| = 0$, there should be a single value of $\beta$ for which \ref{existence2b} holds.  Further, by continuity of the equations of motion, we expect that $\beta$ will again be a continuous function of $\alpha$.\\

\begin{lemma}
\label{Existence3}
Set $\gamma(0)$ as in Equation \ref{initialGamma} for some $E < 0$.  Then there is some $\alpha > 0$ and $\beta > 0$ so that
\begin{equation}
\label{existence3b}
\gamma(\tau) = (a, a, b, c, -c, 0)^T,
\end{equation}
\end{lemma}

This will be proven numerically.  With $\beta = \beta(\alpha)$ as given in Lemma \ref{Existence2}, we show that there is some value of $\alpha$ for which the quantity $P_1(\tau) + P_2(\tau)$ is negative and a second value of $\alpha$ for which $P_1(\tau) + P_2(\tau)$ is positive.  Again, the result follows from the Intermediate Value Theorem. \\

As a consequence of Lemma \ref{Existence3}, we have the following
\begin{theorem}
\label{OrbitConstructed}
There exists a periodic orbit corresponding to the Hamiltonian described in Section \ref{secSetting}.
\end{theorem}

\subsection{Numerical Results}
\label{subNumerICond}

Here we establish numerical results leading up to finding the initial conditions for the orbit.  From here on, we assume that $E = -1$.  We begin with a first result which is helpful in accelerating some of the numerical calculations.  This is most easily set in the original, rather than regularized, coordinates.\\

\begin{lemma}
\label{CollisionOrdering}
Suppose at some $t_0$ time we have that $0 < q_i < 0.4q_j$ and $p_i < p_j$.  Let $t_i > t_0$ and $t_j > t_0$ be the first times where $q_i = 0$ and $q_j = 0$, respectively.  If both $t_i$ and $t_j$ are finite, then $t_i < t_j$.
\end{lemma}

\begin{proof}
Define $\tilde{q}(t) = q_i(t) - q_j(t)$ and $\tilde{p}(t) = p_i(t) - p_j(t)$.  We will show that both $\tilde{q}$ and $\tilde{p}$ are negative and have negative derivatives with respect to time.  Then the region with $\tilde{q} < 0$ and $\tilde{p} < 0$ is forward-invariant (up to collision time), which implies the result.

We first have $\frac{d\tilde{q}}{dt} = p_i(t) - p_j(t) < 0$.  By assumption, $\tilde{q}(t_0) < 0$.  On the other hand, using the Hamiltonian relation gives
\begin{equation}
\frac{d\tilde{p}}{dt} = \frac{4(q_j - q_i)}{(q_i^2+q_j^2)^{3/2}} + \frac{4q_j}{(q_j^2+q_k^2)^{3/2}} - \frac{4q_i}{(q_i^2+q_k^2)^{3/2}} + \frac{1}{2q_j^2} - \frac{1}{2q_i^2}
\end{equation}
Set $q_i = aq_j$ and $q_k = bq_j$.  The assumption $0 < q_i < q_j$ implies that $a \in [0,1]$.  On the other hand, $b$ can be any positive real number.  Making these substitutions and simplifying gives
\begin{equation}
\frac{1}{q_j^2}\left(\frac{4(1-a)}{(a^2+1)^{3/2}} + \frac{4}{(b^2+1)^{3/2}} - \frac{4a}{(a^2+b^2)^{3/2}} + \frac{1}{2} - \frac{1}{2a^2} \right).
\end{equation}
As $q_j^{-2}$ is positive, we note that as $a \to 0^+$, this quantity is negative. \\

Define
\begin{equation}
f(a,b) = \frac{4(1-a)}{(a^2+1)^{3/2}} + \frac{4}{(b^2+1)^{3/2}} - \frac{4a}{(a^2+b^2)^{3/2}} + \frac{1}{2} - \frac{1}{2a^2}.
\end{equation}
A plot of the region where $f(a,b) < 0$ is given in Figure \ref{abregion}. \\

\begin{figure}[h]
\includegraphics[scale=.25]{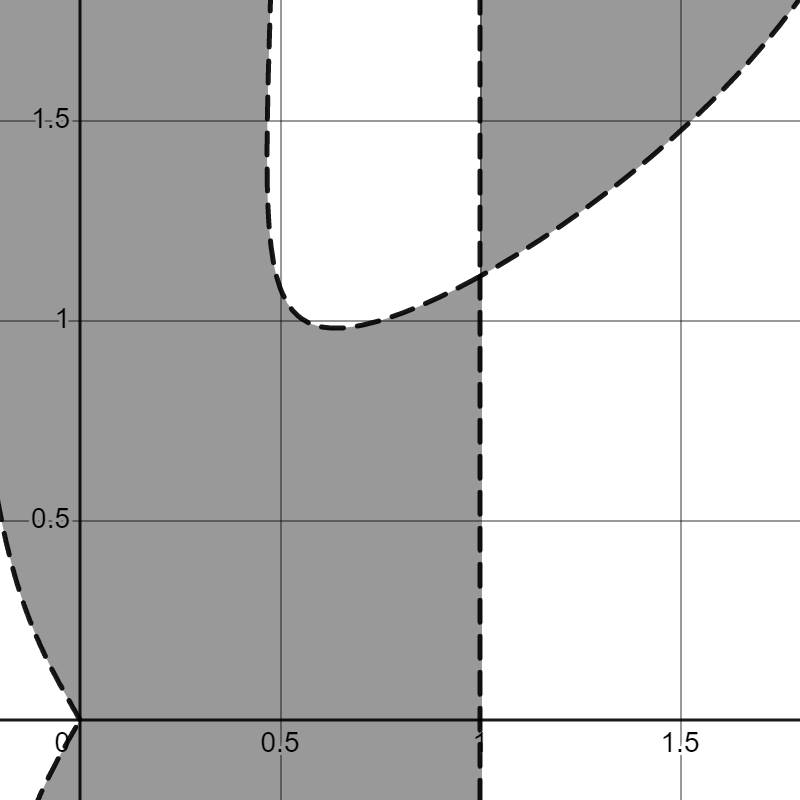}
\caption{A plot of the region where $f(a,b) < 0$.  Here $a$ lies on the horizontal axis and $b$ on the vertical.}
\label{abregion}
\end{figure}

Direct analysis of this function is difficult.  To find the equations of the vertical asymptotes, we consider the equation $f(a,b)=0$ and take the limit as $b \to \infty$.  This yields
\begin{equation}
\frac{4(1-a)}{(a^2+1)^{3/2}} + \frac{1}{2} - \frac{1}{2a^2} = 0.
\end{equation}
With suitable re-arrangement, this can be converted to an eighth-degree polynomial.  We find a numerical estimate of the real root satisfying $0 < a < 1$ is approximately $a \approx 0.523143$.  \\

Numerically we find that for any value of $0 < a < 0.4$, we have $f(a,b)$ (and consequently $\frac{d\tilde{p}}{dt}$) are both less than zero.  Moreover, it is a simple exercise to show that if $\frac{d\tilde{p}}{dt} < 0$ and $0 < a < 0.4$, then the inequality on $a$ is maintained up to collision time, establishing the result.
\end{proof}

We note that in Lemma \ref{CollisionOrdering}, there are other ways to guarantee the forward-invariance of the $\tilde{q} < 0$, $\tilde{p} < 0$ region which include placing inequalities relating to $b$.  However, for simplicity of writing computer code, the version presented is preferred. \\

We can find the value of $\alpha^*$ as defined in Lemma \ref{alphastarlemma} by integrating initial conditions as defined in Equation \ref{initialGamma} and using Lemma \ref{CollisionOrdering} to determine cutoff conditions.  Doing so, we find $\alpha^* = 3.53652$.  We can further use Lemma \ref{CollisionOrdering} to find the boundary of the region $\Sigma$.  Specifically, once we identify two points $(\alpha_1, \beta_1) \in \Sigma$ and $(\alpha_2, \beta_2) \not \in \Sigma$, then by the Intermediate Value Theorem a boundary point of $\Sigma$ must lie on the line segment connecting the two points.  As locating this boundary involves finding initial conditions for an unregularized quadruple collision, Lemma \ref{CollisionOrdering} is vital to reduce the calculation time.  Results of this are shown in Figure \ref{SigmaRegion}.\\

\begin{figure}[h]

\includegraphics[scale=.5]{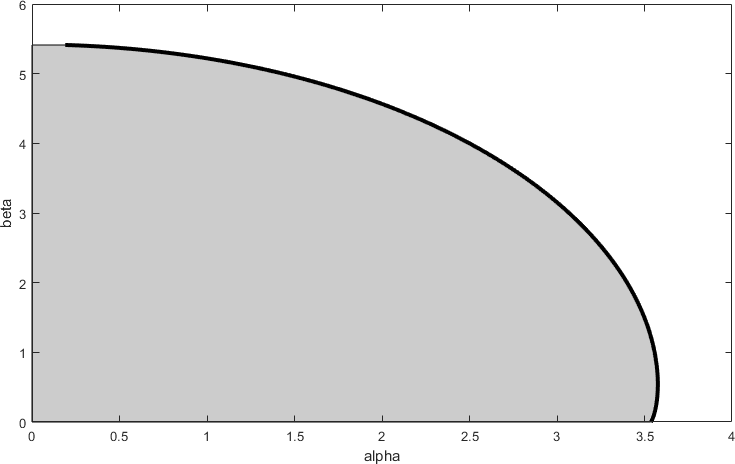}
\caption{The result of a numerical calculation of the region $\Sigma$ with cutoffs based on Lemma \ref{CollisionOrdering}.  The solid black line is the computed curve for which a quadruple collision $Q_1 = Q_2 = 0$ occurs.  For values of $\alpha$ near 0, numerical calculation was terminated before the conditions established by Lemma \ref{CollisionOrdering} were achieved.  An extrapolation of the boundary is pictured.}
\label{SigmaRegion}
\end{figure}

Starting at $\alpha = 0.5, 0.6, 0.7, ...$, we use a bisection method to find the value of $\beta$ as defined in Lemma \ref{Existence2}.  If $Q_3(\tau) > 0$ at the point $(\alpha, \beta)$ corresponding to quadruple collision, then Lemma \ref{Existence3} guarantees the existence of a point $(\alpha, \beta) \in \Sigma$ that corresponds to $P_3(\tau) = 0$.  Accordingly, with initial conditions as defined in Equation \ref{initialGamma}, we integrate the equations of motion using a fixed step-size Runge-Kutta method until the conditions of Lemma \ref{Existence1} are satisfied.  We then vary $\beta$ to find the value for which $P_3(\tau) = 0$.  The results are graphed in Figure \ref{alphabetagraph}.\\

\begin{figure}[h]
\includegraphics[scale=.5]{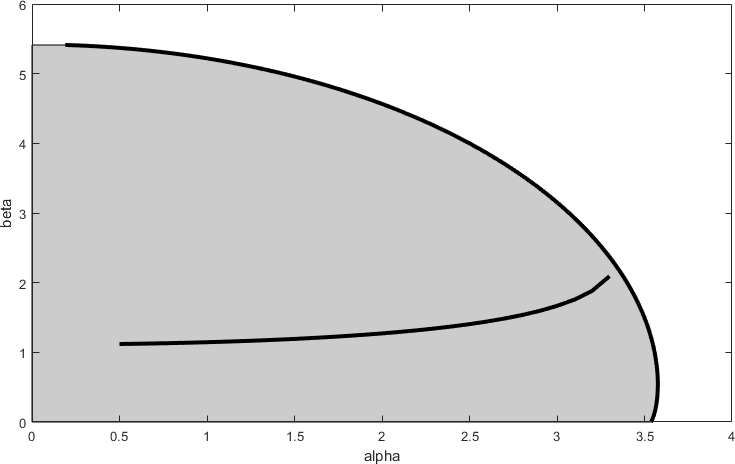}
\caption{The value of $\beta$ (vertical axis) plotted against $\alpha$ (horizontal axis) for which Lemma \ref{Existence2} is satisfied.  The points plotted correspond to $\alpha \in \{.5, .6, ..., 3.3\}$.  For $\alpha = 3.4$, the point $(\alpha, \beta)$ lies outside the region $\Sigma$.}
\label{alphabetagraph}
\end{figure}

Next, along this curve, we compute the value $P_1(\tau) + P_2(\tau)$.  The results are graphed in Figure \ref{pxpygraph}.  For $\alpha = 0.5$, we have $P_1(\tau) + P_2(\tau) = 3.475$, and for $\alpha = 3.3$ we have $P_1(\tau) + P_2(\tau) = -4.112$, establishing Lemma \ref{Existence3}. \\

\begin{figure}[h]
\includegraphics[scale=.5]{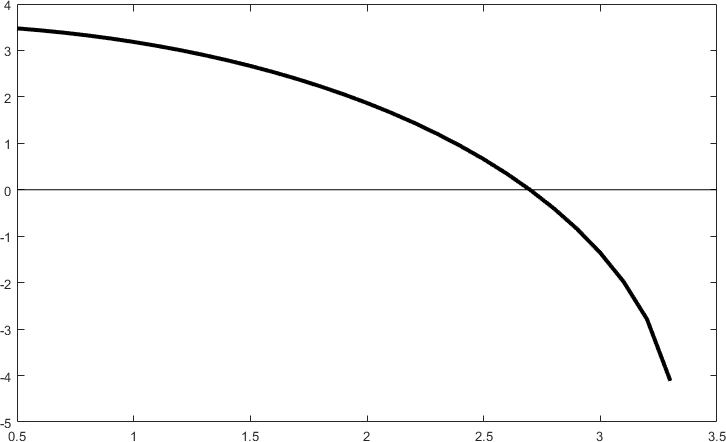}
\caption{The value of $P_1(\tau)+P_2(\tau)$ (vertical axis) plotted against $\alpha$ (horizontal axis) where $\beta = \beta(\alpha)$ from Lemma \ref{Existence2}.}
\label{pxpygraph}
\end{figure}

We then repeatedly refine the value of $\alpha$ around the value closest to $P_1(\tau) + P_2(\tau) = 0$, which is around $\alpha = 2.7$, and repeat the process, working to find the zero of $P_1(\tau) + P_2(\tau)$.  We find the appropriate values are
\begin{equation}
\alpha = 2.698372, \quad \beta = 1.484464.
\label{alphabeta}
\end{equation}
The full period of the regularized orbit is given by
\begin{equation}
12\tau = 0.527482.
\label{periodlength}
\end{equation}
A graph of numerical integration of the regularized equations of motion is given in Figure \ref{RegularizedIntegration}.

\begin{figure}[h]
\includegraphics[scale=.5]{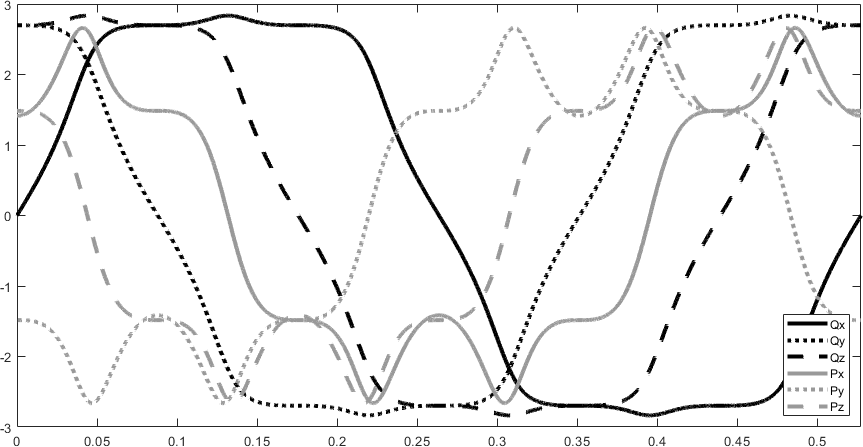}
\caption{Integration of the regularized equations of motion with the initial conditions given by Equation \ref{initialGamma}, and values of $\alpha$ and $\beta$ from Equation \ref{alphabeta}.}
\label{RegularizedIntegration}
\end{figure}

\section{Stability and Symmetry}
\label{secStabSim}

\subsection{Definitions and Preliminaries}
\label{subStabDef}
Let $\mathcal{O}(\gamma_0)$ be the set of all points in $\mathbb{R}^6$ traced out in forward and backward time by the solution to the regularized Hamiltonian $\Gamma$ with initial conditions $\gamma_0$.  If we use the initial conditions determined by $\alpha$ and $\beta$ in the previous section, then the time interval $0 \leq s \leq 12\tau$ captures the entire orbit and $\mathcal{O}(\gamma_0)$ is a closed loop in $\mathbb{R}^6$.  This orbit is \textit{Poincar\'e stable} if given any $\epsilon > 0$ there is some $\delta > 0$ so that for initial conditions $\tilde{\gamma}_0$ with $|\tilde{\gamma}_0 - \gamma_0| < \delta$, then any point on the orbit $\mathcal{O}(\tilde{\gamma}_0)$ is within $\epsilon$ of a point on the orbit $\mathcal{O}(\gamma_0)$. \\

Poincar\'e stability is generally difficult to establish in all but the simplest cases.  However, there is a necessary condition that can be computed.  Specifically, for a Hamiltonian system with Hamiltonian $\Gamma$ and a periodic orbit $\gamma(s)$ with period $T$, consider the matrix differential equation
\begin{equation}
\label{LinearStabilityEquation}
X' = JD^2\Gamma(\gamma(s)), \quad X(0) = I,
\end{equation}
where $D$ denotes the derivative and $J$ is the symplectic matrix
\begin{equation}
J = 
\begin{bmatrix}
0 & I \\
-I & 0
\end{bmatrix}
,
\end{equation}
with $I$ and $0$ are appropriately sized identity and zero matrices respectively.  Then the \textit{monodromy matrix} of the orbit is the matrix $X(T)$, and the orbit is \textit{linearly stable} if the eigenvalues of $X(T)$ all have complex modulus 1 and all have multiplicity one, apart from pairs of eigenvalues equal to 1 corresponding to first integrals of the system.  \\

Linear stability can be established by considering conditions other that $X(0) = I$ as well.  Specifically, if we let $X(0) = Y_0$ be the initial condition to Equation \ref{LinearStabilityEquation}, then $Y(s) = X(s)Y_0$, so $X(T) = Y(T)Y_0^{-1}$.  Hence, $Y_0^{-1}Y(T)$ is similar to the monodromy matrix $X(T)$, and linear stability can be determined from either matrix as similarity preserves eigenvalues. \\

In the octahedral setting, our choice of coordinates has already forced the integrals corresponding to center of mass, net momentum, and angular momentum to be zero.  Hence the monodromy matrix corresponding to the periodic octahedral orbit should contain one pair of eigenvalues 1 corresponding to the fixed value of the Hamiltonian.  Further, it will be shown that a particular choice of $Y_0$ simplifies the calculation.\\

\subsection{Symmetries of the Orbit}
\label{subSym}
A technique by Roberts (see \cite{bibRoberts8}) allows us to further simplify this calculation by ``factoring'' the monodromy matrix in terms of the symmetries of the orbit.  We establish the symmetries of the orbit in this section. \\

By construction of the orbit as given in Lemma \ref{symmetryconstruct}, we have that
\begin{equation}
\gamma(s + 2\tau) = S_f\gamma(s),
\end{equation}
where
\begin{equation}
\label{sfequation}
S_f =
\left[
\begin{array}{ccc|ccc}
0 & 0 & 1 & 0 & 0 & 0 \\
-1 & 0 & 0 & 0 & 0 & 0 \\
0 & 1 & 0 & 0 & 0 & 0 \\
\hline
0 & 0 & 0 & 0 & 0 & 1 \\
0 & 0 & 0 & -1 & 0 & 0 \\
0 & 0 & 0 & 0 & 1 & 0
\end{array}
\right].
\end{equation}
So $S_f$ is a time-preserving symmetry of the orbit.  This symmetry corresponds to a $120^\circ$ rotation about the line $x = y = z$, coupled with an appropriate sign change which arises in the regularized setting. \\

A time-reversing symmetry of the orbit is given by
\begin{equation}
\gamma(-s + 2\tau) = S_r\gamma(s),
\end{equation}
where
\begin{equation}
\label{srequation}
S_r =
\left[
\begin{array}{ccc|ccc}
0 & 1 & 0 & 0 & 0 & 0 \\
1 & 0 & 0 & 0 & 0 & 0 \\
0 & 0 & 1 & 0 & 0 & 0 \\
\hline
0 & 0 & 0 & 0 & -1 & 0 \\
0 & 0 & 0 & -1 & 0 & 0 \\
0 & 0 & 0 & 0 & 0 & -1
\end{array}
\right].
\end{equation}
This can be proven using a similar technique as shown in Lemma \ref{symmetryconstruct}.  Setting $s = \tau$ gives $\gamma(\tau) = S_r\gamma(\tau)$, implying that $\gamma(\tau)$ is an eigenvector of $S_r$ with eigenvalue 1.  Directly computing this, we find that $\gamma(\tau)$ must be of the form
\begin{equation}
\gamma(\tau) = (a, a, b, c, -c, 0)^T
\end{equation}
for suitable values of $a$, $b$, and $c$.  Note that this is exactly what was found in Lemma \ref{symmetryconstruct2}.  \\

\subsection{Roberts's Symmetry-Reduction Technique}
\label{subRoberts}

The general results in this section are presented, with proof, in Section 2 of \cite{bibRoberts8}.  Statements of the results are included here for convenience.  The application of each result to the octahedral orbit is given after each statement.  Results similar to Lemmas \ref{Wform}-\ref{KEig} also appear in \cite{bibRoberts8}, but the form presented in this section is specifically applied to the octahedral orbit.

\begin{lemma}[Lemma 2.1 from \cite{bibRoberts8}]
\label{Roberts21}
Suppose that $\gamma(s)$ is a $T$-periodic solution of a Hamiltonian system with Hamiltonian $\Gamma$ and time-preserving symmetry $S$ such that
\begin{enumerate}
\item There exists some $N \in \mathbb{N}$ such that $\gamma(s + T/N) = S\gamma(s)$ for all $s$,
\item $\Gamma(Sx) = \Gamma(x)$,
\item $SJ = JS$, and
\item $S$ is orthogonal.
\end{enumerate}
Then the fundamental matrix solution $X(s)$ to the linearization problem $\dot{X} = JD^2\Gamma(\gamma(s))X$ with $X(0) = I$ satisfies
\begin{equation}
X(s + T/N) = SX(s)S^TX(T/N).
\end{equation}
\end{lemma}

We note that the matrix $S = S_f$ from Equation \ref{sfequation} satisfies all of these hypotheses with $T = 12\tau$ and $N = 6$. \\

\begin{corollary}[Corollary 2.2 from \cite{bibRoberts8}]
\label{Roberts22}
Given the hypotheses of Lemma \ref{Roberts21}, the fundamental matrix solution $X(s)$ satisfies
\begin{equation}
X(kT/N) = S^k(S^TX(T/N))^k
\end{equation}
for any $k \in \mathbb{N}$.
\end{corollary}

In the case of the octahedral orbit, this gives us that $X(12\tau) = (S_f^TX(2\tau))^6$, as $S_f^6 = I$.

\begin{lemma}[Lemma 2.4 from \cite{bibRoberts8}]
\label{Roberts24}
Suppose that $\gamma(s)$ is a $T$-periodic solution of a Hamiltonian system with Hamiltonian $\Gamma$ and time-reversing symmetry $S$ such that
\begin{enumerate}
\item There exists some $N \in \mathbb{N}$ such that $\gamma(-s + T/N) = S\gamma(s)$ for all $s$,
\item $\Gamma(Sx) = \Gamma(x)$,
\item $SJ = -JS$, and
\item $S$ is orthogonal.
\end{enumerate}
Then the fundamental matrix solution $X(s)$ to the linearization problem $\dot{X} = JD^2\Gamma(\gamma(s))X$ with $X(0) = I$ satisfies
\begin{equation}
X(-s + T/N) = SX(s)S^TX(T/N).
\end{equation}
\end{lemma}

The matrix $S = S_r$ from Equation \ref{srequation} satisfies all of these hypotheses with $T = 12\tau$ and $N = 6$.

\begin{corollary}[Corollary 2.5 from \cite{bibRoberts8}]
\label{Roberts25}
Given the hypotheses of Lemma \ref{Roberts24}, 
\begin{equation}
X(T/N) = SA^{-1}S^TA, \quad A = X(T/2N).
\end{equation}
\end{corollary}

In the case of the octahedral orbit, noting that $S_r^T = S_r$ gives $X(2\tau) = S_rA^{-1}S_rA$ with $A = X(\tau)$.  Combining this with the earlier result, this gives us that the monodromy matrix of the octahedral orbit is $X(12\tau) = (S_f^TS_rA^{-1}S_rA)^6$.  Hence, we can evaluate the stability of the orbit by evaluating the relevant differential equations along only a twelfth of the orbit. \\

Roberts also gives similar results for the case where the initial conditions given in Equation \ref{LinearStabilityEquation} are not the identity matrix.  These are listed below.

\begin{corollary}[Remark following Corollary 2.2 in \cite{bibRoberts8}]
\label{RobertsRemark1}
If $Y(s)$ is the fundamental matrix solution with $X(0) = Y_0$, then
\begin{equation}
Y(s + T/N) = SY(s)Y_0^{-1}S^TY(T/N),
\end{equation}
and so
\begin{equation}
Y(kT/N) = S^kY_0(Y_0^{-1}S^TY(T/N))^k
\end{equation}
\end{corollary}

\begin{corollary}[Remark following Corollary 2.5 in \cite{bibRoberts8}]
\label{RobertsRemark2}
If $Y(s)$ is the fundamental matrix solution with $X(0) = Y_0$, then
\begin{equation}
Y(-s + T/n) = SY(s)Y_0^{-1}S^TY(T/N),
\end{equation}
and so
\begin{equation}
Y(T/N) = SY_0B^{-1}S^TB, \quad B = Y(T/2N)\textcolor{red}{.}
\end{equation}
\end{corollary}

Combining these with previous results gives that for an arbitrary $X(0) = Y_0$, the resulting matrix solution $Y(s)$ to Equation \ref{LinearStabilityEquation} satisfies
\begin{equation}
Y(12\tau) = Y_0(Y_0^{-1}S_f^TS_rY_0B^{-1}S_rB)^6,
\end{equation}
so
\begin{equation}
X(12\tau) = Y_0(Y_0^{-1}S_f^TS_rY_0B^{-1}S_rB)^6Y_0^{-1},
\end{equation}
where $B = Y(\tau)$. \\

Define $W = Y_0^{-1}S_f^TS_rY_0B^{-1}S_rB$.  Then $X(12\tau) = Y_0W^6Y_0^{-1}$, and stability of the octahedral orbit is thus reduced to determining the eigenvalues of $W$. \\

For a properly chosen initial condition matrix $Y_0$, some additional simplification of the calculation can be done.  
\begin{lemma}[Lemma 4.1 from \cite{bibRoberts8}]
Suppose that $W$ is a symplectic matrix satisfying
\begin{equation}
\frac{1}{2}(W + W^{-1}) =
\begin{bmatrix}
K & 0 \\
0 & K^T
\end{bmatrix}.
\end{equation}
Then $W$ has all eigenvalues on the unit circle if and only if the eigenvalues of $K$ lie in the real interval $[-1, 1]$.
\label{Wform}
\end{lemma}
Proper choice of the matrix $Y_0$ will give $W$ of the required form.

\begin{lemma}
Setting $\delta = \sqrt{2}/2$ and 
\begin{equation}
Y_0 = 
\left[
\begin{array}{ccc|ccc}
1 & 0 & 0 & 0 & 0 & 0 \\
0 & 0 & -\delta & 0 & \delta & 0 \\
0 & 0 & \delta & 0 & \delta & 0 \\
\hline
0 & 0 & 0 & 1 & 0 & 0 \\
0 & -\delta & 0 & 0 & 0 & -\delta \\
0 & -\delta & 0 & 0 & 0 & \delta
\end{array}
\right]
\label{matrixy0}
\end{equation}
gives a matrix $W$ of the form in Lemma \ref{Wform}.
\end{lemma}

\begin{proof}
Let
\begin{equation}
\Lambda =
\begin{bmatrix}
I & 0 \\
0 & -I \\
\end{bmatrix}
\end{equation}
where $I$ and $0$ represent $3 \times 3$ identity and zero matrices, respectively.  Then direct calculation yields $-Y_0^{-1}S_f^TS_rY_0 = \Lambda$. \\

Set $D = -B^{-1}S_rB$.  Then by definition of $W$ we have that $W = \Lambda D$.  Since $D^2 = \Lambda^2 = I$, then we know that $W^{-1} = D\Lambda$.  Since $B$ is symplectic, writing
\begin{equation}
B =
\begin{bmatrix}
B_1 & B_2 \\
B_3 & B_4
\end{bmatrix}
\quad \text{ and } \quad
S_r =
\begin{bmatrix}
S & 0 \\
0 & -S
\end{bmatrix},
\end{equation}
then the formula for the inverse of a symplectic matrix gives
\begin{equation}
B^{-1} =
\begin{bmatrix}
B_4^T & -B_2^T \\
-B_3^T & B_1^T
\end{bmatrix}.
\end{equation}
Directly computing $D$ gives

\begin{equation}
D =
\begin{bmatrix}
K^T & L_1 \\
-L_2 & -K
\end{bmatrix}
\end{equation}
with $K$, $L_1$, and $L_2$ defined up to sign by matrix multiplication.  Then
\begin{equation}
W = \Lambda D =
\begin{bmatrix}
K^T & L_1 \\
L_2 & K
\end{bmatrix}
\quad \text{ and } \quad
W^{-1} = D \Lambda =
\begin{bmatrix}
K^T & -L_1 \\
-L_2 & K
\end{bmatrix}
\end{equation}
and
\begin{equation}
\frac{1}{2}(W + W^{-1}) =
\begin{bmatrix}
K^T & 0 \\
0 & K
\end{bmatrix}
\end{equation}
as required.
\end{proof}

As noted earlier, our coordinate system has already made use of the first integrals corresponding to center of mass, net momentum, and angular momentum in this setting.  There is an additional pair of eigenvalues $1$ in the monodromy matrix corresponding to the remaining first integral, the Hamiltonian itself.  These can be found, with eigenvector, as shown below. \\

\begin{lemma}
The matrix $K^T$ has a right eigenvector $[1 \ 0 \ 0]^T$ with corresponding  eigenvalue $1$.
\label{KEig}
\end{lemma}
\begin{proof}
Let $v = Y_0^{-1}\gamma'(0)/||\gamma'(0)|| = Y_0^{T}\gamma'(0)/||\gamma'(0)||$.  Since $Y_0$ is orthogonal and $S_r$ is symmetric, we have
\begin{equation}
W = Y_0^{-1}S_f^TS_rY_0B^{-1}S_rB = Y_0^{T}S_f^TS_rY_0B^{-1}S_r^TB = Y_0^TS_f^TY(2\tau)
\end{equation}
by Corollary \ref{RobertsRemark2} with $s = 0$. \\

Define $\gamma(s)$ to be the periodic orbit with initial conditions defined in Section \ref{secPeriodicOrbit}.  Since $\gamma'(s)$ is a solution to $\dot{\xi} = JD^2\Gamma(\gamma(s))\xi$ and $$\gamma'(0) = Y(0)Y_0^{-1}\gamma'(0) = Y(0)v,$$
then
\begin{equation}
\gamma'(s) = Y(s)Y_0^{-1}\gamma'(0) = Y(s)v.
\end{equation}
This implies
\begin{equation}
Y_0^{-1}S_f^T\gamma'(2\tau) = Y_0^TS_f^TY(2\tau)v = Wv.
\end{equation}
Since 
\begin{equation}
\gamma'(0) = (2\sqrt{2}\alpha^4, 0, 0, 0, 0, 0)
\end{equation}
and
\begin{equation}
\gamma'(2\tau) = (0, -2\sqrt{2}\alpha^4, 0, 0, 0, 0)
\end{equation}
with $\alpha$ as defined in Equation \ref{alphabeta}, we have
\begin{equation}
S_f^T\gamma'(2\tau) = \gamma'(0).
\end{equation}
Then $$Wv = Y_0^{-1}S_f^T\gamma'(2\tau) = Y_0^TS_f^TS_f\gamma'(0) = Y_0^TY_0v = v.$$
So $v$ is an eigenvector of $W$ with eigenvalue $1$. \\

Since $\gamma'(0)$ is known, we have that $v = Y_0^{-1}e_1$, where
\begin{equation}
e_1 = [1 \ 0 \ 0 \ 0 \ 0 \ 0]^T.
\end{equation}
Direct calculation gives that $v = e_1$.  Then, since $W$ satisfies the relation given in Lemma $\ref{Wform}$, $K^T$ must have eigenvector $[1 \ 0 \ 0]^T$ with eigenvalue $1$ as claimed.
\end{proof}

As a consequence, we know that the matrix $K$ must be of the form
\begin{equation}
K =
\begin{bmatrix}
1 & 0 & 0 \\
* & k_{22} & k_{23} \\
* & k_{32} & k_{33}
\end{bmatrix}
\end{equation}
and so the eigenvalues of the lower-right $2 \times 2$ block will determine stability.  Because of the form of $K$, entries marked with an $*$ are not needed for the stability calculation.

\section{Stability Results}
\label{secStabRes}

Using the matrix $Y_0$ from Equation \ref{matrixy0}, we find the matrix $B = Y(\tau)$ numerically with the initial conditions from Equation \ref{alphabeta}. Then the matrix $K$ is given numerically by
\begin{equation}
K = 
\begin{bmatrix}
0.99998 & -0.00002 & -0.00000 \\
-0.33061 & 0.46342 & 0.04988 \\
0.90212 & -1.96273 & -1.28477
\end{bmatrix}
\end{equation}
The values given for the first row entries are the result of propagation of numerical error in the calculation.  However, the small deviation from the values $1$ and $0$ gives confidence that the numerical calculation is correct.  Assuming the computed values are $1$ and $0$ zero as proven earlier, the eigenvalues from the lower-right $2 \times 2$ block of $K$ are given by a simple application of the quadratic formula.  We find
\begin{equation}
\lambda_1 = 0.40550, \quad \lambda_2 = -1.22685
\end{equation}
As a consequence of Lemma \ref{Wform}, we have the following
\begin{theorem}
The octahedral orbit described throughout this paper is linearly unstable.
\end{theorem}

\section{Data Statement}

All data generated or analyzed during this study are included in this published article.

\section{Acknowledgements}

This section will be completed later.  Any content here will not affect the core mathematical content presented in the prior sections.

\bibliographystyle{plain}

\bibliography{OctahedralNotesDraft2}

\end{document}